\documentclass[12pt,leqno,a4paper]{amsart}
\usepackage{amssymb,enumerate}
\newcommand{\FF}{{\mathbb{F}}}

\newcommand{\fA}{{\mathfrak{A}}}
\newcommand{\fS}{{\mathfrak{S}}}

\newcommand{\Hom}{{\operatorname{Hom}}}
\newcommand{\Ind}{{\operatorname{Ind}}}
\newcommand{\Infl}{{\operatorname{Infl}}}
\newcommand{\Irr}{{\operatorname{Irr}}}
\newcommand{\Res}{{\operatorname{Res}}}
\newcommand{\St}{{\operatorname{St}}}

\newcommand{\GAP}{{\sf{GAP}}}
\newcommand{\Chevie}{{\sf{Chevie}}}

\newcommand{\GL}{{\operatorname{GL}}}
\newcommand{\PGL}{{\operatorname{PGL}}}
\newcommand{\SL}{{\operatorname{SL}}}
\newcommand{\PSL}{{\operatorname{L}}}
\newcommand{\GU}{{\operatorname{GU}}}

\newcommand{\SU}{{\operatorname{SU}}}
\newcommand{\PSU}{{\operatorname{U}}}
\newcommand{\Sp}{{\operatorname{Sp}}}
\newcommand{\PSp}{{\operatorname{S}}}
\newcommand{\GO}{{\operatorname{GO}}}
\newcommand{\SO}{{\operatorname{SO}}}
\newcommand{\OO}{{\operatorname{O}}}

\newcommand{\tw}[1]{{}^#1\!}

\let\eps=\epsilon
\let\la=\lambda

\newtheorem{thm}{Theorem}[section]
\newtheorem{lem}[thm]{Lemma}

\newtheorem{cor}[thm]{Corollary}
\newtheorem{prop}[thm]{Proposition}

\newtheorem{thmA}{Theorem}

\theoremstyle{definition}
\newtheorem{exmp}[thm]{Example}

\theoremstyle{remark}

\begin{document}

\title{Projective indecomposable permutation modules}

\date{\today}

\author{Gunter Malle}
\address{FB Mathematik, TU Kaiserslautern, Postfach 3049,
         67653 Kaisers\-lautern, Germany.}
\email{malle@mathematik.uni-kl.de}
\author{Geoffrey R. Robinson}
\address{Institute of Mathematics, University of Aberdeen, Aberdeen AB24 3UE,
  Scotland, United Kingdom}
\email{g.r.robinson@abdn.ac.uk}

\thanks{The first author gratefully acknowledges support by the Deutsche
 Forschungsgemeinschaft -- Project-ID 286237555-- TRR 195.}

\dedicatory{To Pham Huu Tiep on the occasion of his 60th birthday}



\begin{abstract}
We investigate finite non-Abelian simple groups $G$ for which the projective
cover of the trivial module coincides with the permutation module on a
subgroup and classify all cases unless $G$ is of Lie type in defining
characteristic.
\end{abstract}

\maketitle


\section{Introduction}   \label{sec:intro}

Let $G$ be a finite group, $p$ a prime and $k$ an algebraically closed field of
characteristic~$p$. We are interested in the situation when the projective cover
$\Phi_{1_G}$ of the trivial $kG$-module $k$ is the permutation module on a
subgroup $H$ of $G$. We then say that $G$ \emph{has property $(I_p)$ with
respect to the subgroup $H$}. Note that in this case $H$ is necessarily of
$p'$-order. Thus, $G$ has property $(I_p)$ if there is a $p'$-subgroup $H$ of
$G$ such that the endomorphism ring of the permutation module on $H$ is a local
ring. Of course, $G$ has property $(I_p)$ for every prime $p$ not dividing
its order (with respect to $H=G$), so the interesting case is for the prime
divisors of $|G|$.

We investigate finite non-Abelian simple groups $G$ enjoying property
$(I_p)$ for some prime~$p$. This turns out to be quite a rare phenomenon.

\begin{thmA}   \label{thm:main}
 Let $G$ be a non-Abelian finite simple group with property $(I_p)$. Then
 either the pair $(G,p)$ is classified below, or $G$ is of Lie type in
 characteristic~$p$, or $G=J_4$ with $p=11$.
\end{thmA}

In all cases coming up in our classification either $p\le3$, or Sylow
$p$-subgroups of $G$ are of a very restricted form: they are either cyclic, or
Abelian of rank~2, or an extra-special group $7^{1+2}$.

The open cases are proof of the little that is known (to the authors) about
decomposition numbers of $J_4$ and, in particular, of groups of Lie type in
defining characteristic. For the exceptional groups of small rank we do obtain
complete results in Section~\ref{sec:defchar}, though.
\vskip 1pc

After collecting some elementary observations in Section~\ref{sec:prelim} we
give the proof of Theorem~\ref{thm:main} by dealing with the various classes of
finite non-Abelian simple groups according to the classification in
Sections~\ref{sec:alt}--\ref{sec:defchar}. Our approach relies on detailed
information on their maximal subgroups of large order.

\section{General considerations}   \label{sec:prelim}

We collect some elementary observations on projective indecomposable
permutation modules.
Willems \cite{Wi80} has studied the 1-PIM for composite groups and obtained
several reduction results which show that the main question is for non-Abelian
simple groups; see in particular Lemma~\ref{lem:normal} below.

\begin{lem}   \label{lem:order}
 Assume $G$ has property $(I_p)$ with respect to $H\le G$. Then, for any
 $p'$-subgroup $L\le G$ and any $\chi\in\Irr(G)$ we have
 $$\langle\Res_H^G(\chi),1_H\rangle \leq \langle\Res_L^G(\chi),1_L\rangle.$$
 In particular, $|H|$ is maximal among all $p'$-subgroups of $G$.
\end{lem}

\begin{proof}
We note that the projective cover of the trivial $kG$-module is a direct
summand of $\Ind_L^G(k)$ (with trivial action of $L$ on $k$). Hence the virtual
character afforded by $\Ind_L^G(k)-\Ind_H^G(k)$ is in fact a character. The
claimed inequality follows by Frobenius reciprocity. Observe that $H$ has to be
of $p'$-order as $\Ind_H^G(k)$ is projective by assumption.
\end{proof}

\begin{lem}   \label{lem:2-trans}
 Assume that $G$ acts $2$-transitively on the set of cosets of $H\le G$. Then
 $G$ has property $(I_p)$ with respect to $H$ for any prime $p$ dividing $G$ but
 not dividing $|H|$.
\end{lem}

\begin{proof}
If $G$ acts 2-transitively on the cosets of $H$ then $\Ind_H^G(1_H)$ has
exactly two ordinary irreducible constituents. If $H$ is a $p'$-group then this
module is projective, and indecomposable since the trivial $kG$-module is not
projective when $p$ divides $|G|$.
\end{proof}

The following is a kind of weak converse:

\begin{lem}   \label{lem:cyc}
 Let $G$ be a finite group with cyclic Sylow $p$-subgroups. Assume that the
 trivial character is not connected to the exceptional node on the $p$-Brauer
 tree of $G$. If $G$ has property $(I_p)$ with respect to $H\le G$ then
 $G$ acts $2$-transitively on the set of cosets of~$H$.
\end{lem}

\begin{proof}
By the well-known theory of blocks with cyclic defect, the assumption on the
Brauer tree implies that the projective cover $\Phi_{1_G}$ of the trivial
$kG$-module has just two ordinary constituents. As $\Phi_{1_G}=\Ind_H^G(1_H)$
by assumption, the permutation character of $G$ on the cosets of $H$ has just
two constituents, and so the action is 2-transitive.
\end{proof}

We now note two results that are important for induction purposes, the first of
which is clear:

\begin{lem}   \label{lem:overgroup}
 Let $H\le L\le G$. If $G$ has property $(I_p)$ with respect to $H$ then
 $L$ has property $(I_p)$ with respect to $H$.
\end{lem}

Now let $N\unlhd G$ be a normal subgroup. Then by \cite[Lemma~2.6]{Wi80} we have
$$\dim\Phi_{1_G}=\dim\Phi_{1_{G/N}}\,\dim\Phi_{1_N}.\eqno{(*)}$$

\begin{lem}   \label{lem:normal}
 Assume $G$ has property $(I_p)$ with respect to $H\le G$. Then for any
 normal subgroup $N\unlhd G$, $G/N$ has property $(I_p)$ with respect to $HN/N$
 and $N$ has property $(I_p)$ with respect to $N\cap H$.
\end{lem}

\begin{proof}
Clearly, both $HN/N\cong H/(H\cap N)$ and $H\cap N$ are $p'$-groups, and
by~$(*)$ we have
$$\dim\Phi_{1_{G/N}}\dim\Phi_{1_N}=\dim\Phi_{1_G}=|G:H|
  =|G/N:HN/N|\,|N:H\cap N|.$$
Since $\dim\Phi_{1_{G/N}}\le|G/N:HN/N|$ and $\dim\Phi_{1_N}\le|N:H\cap N|$ the
claim follows.
\end{proof}

While the converse holds, for example, for $p$-solvable groups, since for
these, Hall $p'$-subgroups are the unique conjugacy class of maximal
$p'$-subgroups, it is not true in general: $G=\PSL_2(7)$ has property $(I_7)$
with respect to a subgroup $H=\fS_4$, but $\hat G=\PGL_2(7)$ does not satisfy
$(I_7)$. Similarly, $G=\fA_7$ has property~$(I_5)$ with respect to
$H=\PSL_2(7)$, but $\hat G=\fS_7$ does not satisfy $(I_5)$. In both cases, the
ambient $p'$-subgroup $H$ of $G$ is not stable under the outer automorphism
of~$G$ induced by $\hat G$.

\begin{cor}   \label{cor:compfac}
 Assume $G$ has property $(I_p)$. Then any non-Abelian simple composition factor
 of $G$ satisfies $(I_p)$.
\end{cor}

Let us remark, though, that not every subgroup of a group satisfying $(I_p)$
also does: while $G=\PSL_3(4)$ has property~$(I_3)$, its maximal subgroup
$\fA_6$ does not.

\begin{lem}   \label{lem:fp}
 If $G$ satisfies $(I_p)$ with respect to $H$, then no non-trivial simple
 $kG$-module has $H$-fixed points.
\end{lem}

\begin{proof}
Assume $\Ind_H^G(k)=\Phi_{1_G}$. If $H$ has non-zero fixed points on the
simple $kG$-module $S$, then $\Hom_{kH}(k,\Res^G_H(S))\neq0$, so
$\Hom_{kG}(\Ind_H^G(k),S)) \neq 0$ by Frobenius reciprocity. Since the head
of $\Phi_{1_G}$ is simple, and equal to the trivial $kG$-module $k$, this
implies $S\cong k$.
\end{proof}

We have the following consequence for the number $l(G)$ of irreducible
$p$-Brauer characters of a group $G$ satisfying our requirement:

\begin{lem}   \label{lem:l(G)}
 Let $G$ be a finite group and $p$ a prime. Assume $G$ has property $(I_p)$ with
 respect to $H\le G$. Then $l(G)\le |H|$.
\end{lem}

\begin{proof}
We have $|G| = \sum_S\dim S\cdot\dim P_S$, where the sum runs over the $l(G)$
isomorphism classes of simple $kG$-modules and $P_S$ is the projective cover
of~$S$. But notice that $S^*\otimes P_S$ has the projective cover of the trivial
module as a summand.

Hence we have $|G| \geq l(G)\dim\Phi_{1_G}$. If $\Phi_{1_G}$ is the
permutation module on the cosets of~$H$, then we obtain
$|H|\dim\Phi_{1_G} = |G| \geq l(G)\dim\Phi_{1_G}$, so that $l(G) \leq |H|$, as
claimed.
\end{proof}

In particular the proof shows that assuming $G$ satisfies $(I_p)$ with respect
to $H\le G$, $|G:H|\le\chi(1)^2$ for every $\chi\in\Irr(G)$ of $p$-defect zero.

\section{Alternating groups} \label{sec:alt}

\begin{thm}   \label{thm:alt}
 The alternating group $\fA_n$, $n\ge5$, has property $(I_p)$ for $p\le n$ if
 and only if we are in one of the cases of Table~\ref{tab:alt}.
\end{thm}

\begin{table}[htb]
\caption{Induced 1-PIMs for alternating groups}   \label{tab:alt}
$\begin{array}{cccr}
 G& H& p& \dim\Phi_{1_G}\kern -15pt\\
\hline
 \fA_n&     \fA_{n-1}& n& n\\
 \fA_5&           C_5& 2& 12\\
 \fA_5&           D_5& 3& 6\\
 \fA_6&         C_3^2& 2& 40\\
 \fA_6&         3^2.4& 5& 10\\
 \fA_7&     \PSL_3(2)& 5& 15\\
 \fA_8& 2^3.\PSL_3(2)& 5& 15\\
\end{array}$
\end{table}

All entries in the table are indeed examples: a Sylow $p$-subgroup of $\fA_p$
is cyclic, and the action on $\fA_{p-1}$ is 2-transitive, this gives the
infinite series. The additional examples for $n=5,6,7,8$ can easily be checked
from the decomposition matrices in \GAP\ \cite{GAP}.
The proof of the converse proceeds by considering the various types of
subgroups $H$ of~$\fA_n$.

\begin{prop}   \label{prop:alt trans}
 Assume that $\fA_n$, $n\ge5$, is a minimal counter-example to
 Theorem~\ref{thm:alt} with respect to $H$. Then $H$ is a transitive
 subgroup; in particular, $p{\not|}n$.
\end{prop}

\begin{proof}
Let $n$ be minimal such that $G=\fA_n$ satisfies $(I_p)$ with respect to $H\le G$
not occurring in the conclusion and assume $H$ is intransitive. Let $M<\fA_n$
be a maximal intransitive subgroup containing $H$. Thus
$M=(\fS_k\times\fS_{n-k})\cap\fA_n$ for some $1\le k\le n-1$. If $\fA_n$ has
property~$(I_p)$ then by Lemma~\ref{lem:overgroup} and
Corollary~\ref{cor:compfac} so do $\fA_k$ and $\fA_{n-k}$. If $p>5$ then by
induction we must have $k,n-k\le p\le n$. In fact, by applying
Proposition~\ref{lem:cyc} either $n=p$ or $k=n-k=p$. The first case
appears in the conclusion of
Theorem~\ref{thm:alt}, while in the second case the permutation character of
$\fA_{2p}$ on $M$ contains characters from non-principal $p$-blocks, for example
the restriction to $\fA_{2p}$ of the character labelled by $(2p-1,1)$, so this
does not occur.   \par
If $p=5$ then we need to discuss $k,n-k\le 8$, if $p=3$, then $k,n-k\le5$, and
if $p=2$ then $k,n-k\le 6$. These cases can be settled using \GAP.
\end{proof}

\begin{prop}   \label{prop:alt prim}
 Assume that $\fA_n$, $n\ge5$, is a minimal counter-example to
 Theorem~\ref{thm:alt} with respect to $H$. Then $H$ is a primitive
 subgroup.
\end{prop}

\begin{proof}
Let $n$ be minimal such that $G=\fA_n$ satisfies $(I_p)$ with respect to $H\le G$
not occurring in the conclusion of Theorem~\ref{thm:alt}.
By Proposition~\ref{prop:alt trans}, $H$ is transitive and $p{\not|}n$. Assume
$H$ is imprimitive and let $M<\fA_n$ be a maximal imprimitive overgroup. Then
$M=\fS_a\wr\fS_b\cap\fA_n$ with $ab=n$. By Corollary~\ref{cor:compfac} and
minimality we have $a,b<p$, or $p\le5$. Then, $a,b>2$ as otherwise Sylow
$p$-subgroups of $G$ are Abelian.
Now, if $p=2$ then $a,b\in\{3,5\}$. The cases $n=9,15$ can be excluded with
\GAP, and for $n=25$ the maximal $p'$-subgroup $H=C_5\wr C_5$ of
$M=\fS_5\wr\fS_5\cap\fA_{25}$ has smaller order than that of a Sylow
3-subgroup of $\fA_{25}$, so
$M$ cannot contain a relevant subgroup. If $p=3$ then $a,b\in\{2,4,5\}$. Again,
the cases $n\le16$ are settled using the known tables. When $n=20$ a
subgroup $F_{20}\wr D_8\cap\fA_{20}$ has larger order than the largest
$3'$-subgroup of $\fS_4\wr\fS_5$; when $n=25$ again a Sylow 2-subgroup has
larger order than the $p'$-subgroup $D_5\wr D_5$ of $\fA_5\wr\fA_5$.   \par
If $p=5$ then $a,b\in\{2,3,4,6,7,8\}$, so (using \GAP)
$$n\in\{ 21, 24, 28, 32, 36, 42, 48, 49, 56, 64 \}.$$
Let $r=19,23,23,31,31,41,47,47,53,61$ in the respective cases. Then $M$ is
an $r'$-subgroup of $\fS_n$, so the permutation character on $M$ is
$r$-projective and hence contains the character connected to the trivial
character on the $r$-Brauer tree. This has label
$$(18,3),(22,2),(22,6),(30,2),(30,6),(40,2),(46,2),(46,3),(52,4),(60,4)$$
respectively. Since none of these lies in the principal 5-block, no new
cases arise.
\par
We may now assume $2< a,b<p$, so $n=ab\le(p-1)^2$. Then \cite[Thm~2.6 and
Cor.~6.4]{PW11} exhibits a constituent labelled $(a^b)$, not lying in the
principal $p$-block of $\fA_n$ unless $a=p-1$, so $n=b(p-1)$ with $2<b<p$. Let
$r$ be a prime between $n/2+1$ and $n$. Then Sylow $r$-subgroups of $\fS_n$ are
cyclic. Arguing as above, the permutation character on $M$ must contain the
character connected to the trivial character on the $r$-Brauer tree, labelled
$(r-1,n-r+1)$. This lies in the principal $p$-block only when
$n-r+1\equiv0,1\pmod p$, so $b+r\equiv0,1\pmod p$. Thus, less than $b$ possible
odd values for $r$ are excluded. Now by \cite[Cor.~3]{RS62} the number of
primes between $x$ and $2x$ is at least $3x/(5\log x)$ for $x\ge 21$. This
shows there are at least $b$ distinct primes $r$ in our range as soon as
$p\ge13$. The smaller values of $p$ are readily checked.
\end{proof}

\begin{proof}[Proof of Theorem~\ref{thm:alt}]
Let $n$ be minimal such that $\fA_n$ is a counter-example to the theorem,
with respect to the $p'$-subgroup $H$. By using the tables in \GAP\ we may
assume $n>16$. By Propositions~\ref{prop:alt trans} and~\ref{prop:alt prim},
$H$ is a transitive and primitive subgroup. Hence, for $n>24$ we have
$|H|<2^n$ by a result of Mar\'oti \cite[Cor.1.2]{Ma02}. But for $n\ge7$ the
subgroup $\fS_4^a\fS_b\cap\fA_n$, where $n=4a+b$ with $0\le b<4$, has order
larger than $2^n$. This contradicts our assumption on $H$, by
Lemma~\ref{lem:order}, when $p\ge5$ and $n>24$. The primitive $p'$-subgroups of
$\fA_n$, $n\le24$, are well known and easily seen to be too small as well.
\par
We are left with the case $p\in\{2,3\}$. For $p=3$ consider suitable
direct products of wreath products of the subgroup of order $20$ inside
$\fS_5$. This has order larger than $n^{\sqrt{n}}$ for $n\ge35$, and
hence larger than any primitive non-3-transitive subgroup of $\fA_n$, by
\cite[Cor.~1.1(i)]{Ma02}. Note that 3-transitive groups are not $3'$-groups.
For $n\le 34$ the primitive groups are available in \GAP\ and have order
smaller than that of a Sylow 2-subgroup of~$\fA_n$.

For $p=2$ the order of a Sylow 3-subgroup of $\fA_n$ is larger than
$n^{\sqrt{n}}$ for $n\ge60$, so by the above cited result we only need to
worry about $n\le59$. The odd-order (solvable) primitive subgroups of these
alternating groups have smaller order than a Sylow 3-subgroup of $\fA_n$ by
\GAP.
\end{proof}

\section{Sporadic groups} \label{sec:spor}

\begin{thm}   \label{thm:spor}
 Let $G$ be a sporadic simple group and $p$ a prime dividing $|G|$. If $G$ has
 property $(I_p)$ with respect to $H\le G$ then $(G,H,p)$ are as in
 Table~$\ref{tab:spor}$. Moreover, in all listed cases except possibly $J_4$,
 the groups do have property $(I_p)$.
\end{thm}

\begin{table}[htb]
\caption{Induced 1-PIMs for sporadic groups}   \label{tab:spor}
$\begin{array}{ccrrc}
 G& H& p& \dim\Phi_{1_G}\kern-10pt\\
\hline
     M_{11}&    M_{10}& 11& 11\\
     M_{22}& \PSL_3(4)& 11& 22\\
     M_{23}&    M_{22}& 23& 23\\
         HS& \PSU_3(5)& 11& 176\\
       Co_3&     McL.2& 23& 276\\
        J_2& \PSU_3(3)&  5& 100\\
       He& \PSp_4(4).2&  7& 2058\\
 J_4& 2^{10}.\PSL_5(2)& 11& 8474719242\\
\end{array}$
\end{table}

\begin{proof}
The decomposition matrices of sporadic groups are completely known up to the
Harada--Norton group $HN$ and contained in \cite{GAP}. From this the claim can
be checked easily for those groups by first discarding those cases when
$\dim\Phi_{1_G}$ does not divide $|G|$, and in the few remaining cases, checking
whether a subgroup of the required index exists. Five examples occur for cyclic
Sylow $p$-subgroups, as described by Lemma~\ref{lem:2-trans}, the other two have
non-cyclic, Abelian Sylow $p$-subgroups.   \par
For the larger groups the real stem of the Brauer tree of the principal
$p$-block is known \cite{HL89}. The trivial character is never
connected to the exceptional node, and since these groups do not possess
2-transitive permutation representations by \cite[Thm~5.3]{C81}, there are
no further examples with cyclic Sylow $p$-subgroups by Lemma~\ref{lem:cyc}.

The permutation characters of the large maximal subgroups $U$ of the remaining
sporadic groups $G$ are available in \GAP. Removing those which involve
characters from non-principal blocks, we are left with a small list of
possible cases. For example, for $G=Ly$ we could have $H$ contained in one of
the
two largest maximal subgroups, $G_2(5)$ or $3.McL.2$. But neither of these has
property $(I_p)$, so nor does $G$ by Lemma~\ref{lem:overgroup}. The cases that
can not directly be ruled out like this are: 
$J_4$ at $p=11$ with $U$ one of $2^{13}.3.M_{22}.2$ or $2^{10}.\PSL_5(2)$,
$G=B$ at $p=2,3$ with $U=2.\tw2E_6(2).2$, and $G=M$ at $p=2,3$ with $U=2.B$.

Arguing for $G:=\tw2E_6(2)$ in the same way as we did for the sporadic groups
we see that if it has property $(I_p)$, then either $p=2$, or $p=3$ with
$H\le U:=2^{1+20}.\PSU_6(2)$. Now for $p=3$ the 1-PIM $\bar\Phi_1$ of
$\bar U:=\PSU_6(2)$ is known \cite{GAP}, and
$\Ind_U^G(\Phi_{1_U})=\Ind_U^G(\Infl_{\bar U}^U(\bar\Phi_1))$ is just
Harish-Chandra induction of $\bar\Phi_1$, hence can be computed explicitly.
It transpires that this contains the unipotent constituent labelled
$\phi_{16,5}$ which does not lie in the principal 3-block \cite{En00}. So this
does not yield an example for $\tw2E_6(2)$ at $p=3$.  For $p=2$, using
Lemma~\ref{lem:normal} and the information in \GAP\ the only maximal subgroups
of $G$ possibly containing a relevant $2'$-subgroup
$H$ are the parabolic subgroup with Levi factor $\fA_5\times\PSL_3(2)$, and
two 3-local subgroups. The 3-local subgroup $U$ normalising a 3C-element
contains a subgroup of order~$3^9$, larger than the $2'$-part of either of the
other two candidate subgroups. Hence we must have $H\le U$, up to conjugation.
But the permutation character of $G$ on $U$
contains characters not in the principal 2-block. So $p=2$ is not possible for
$\tw2E_6(2)$ either, and hence we also do not obtain examples for $B$ or $M$
at~$p=2$ or $p=3$ by Lemma~\ref{lem:overgroup}.   \par
For $J_4$ the $11'$-part of the order of the maximal subgroup
$2^{13}.3.M_{22}.2$ is smaller than the order of the maximal subgroup
$2^{10}.\PSL_5(2)$, which is not divisible by~11, so if $J_4$ is an example
then with respect to the latter subgroup.
\end{proof}

For $G=J_4$ with $p=11$ the maximal subgroups $2^{10}.\PSL_5(2)$ are at the
same time maximal $p'$-subgroup. We have not been able to settle this case.

\section{Groups of Lie type in non-defining characteristic} \label{sec:nondef}

We now consider the simple groups of Lie type $G$ in characteristic~$r$ for
primes $p\ne r$ dividing $|G|$. Our general strategy is as follows. By
Lemma~\ref{lem:order} any admissible $p'$-subgroup $H$ of $G$ has order at
least that of a Sylow $r$-subgroup of $G$. Most maximal subgroups of $G$
with at least that order have been determined by Liebeck and Liebeck--Saxl.
According to \cite[Thm]{LS87} and \cite[Thm]{Lie85}, these are either
maximal parabolic subgroups, or some narrow class of subsystem subgroups.
Let's first discuss the former. There are two main arguments:
\begin{enumerate}
\item The permutation characters on parabolic subgroups $P<G$ are known by
 Howlett--Lehrer theory to decompose as the corresponding characters in the
 associated Weyl groups. All of their constituents are unipotent. From the known
 block distribution of unipotent characters \cite{BMM} we can determine
 whether some constituents of $\Ind_P^G(1_P)$ do not lie in the principal
 $p$-block of $G$. In that case, $P$ cannot contain a candidate subgroup $H$.
\item Let $P=UL$ be the Levi decomposition of a parabolic subgroup $P$ of $G$,
 with maximal normal unipotent
 subgroup~$U$. If $H$ is contained in $P$, then $P$ has property $(I_p)$ by
 Lemma~\ref{lem:overgroup}, and so has $[L,L]/Z([L,L])$ by
 Corollary~\ref{cor:compfac}. At least if $q>3$ this is a product of simple
 groups of Lie type of smaller rank, for which we know about validity of
 $(I_p)$ by induction.
\end{enumerate}
As for the exceptions in the Liebeck and Liebeck--Saxl theorems, these are
either for specific, small values of $q$, in which case decomposition matrices
in the \GAP-library \cite{GAP} can be used. For the remaining maximal subgroups
$M$, we try to
either exhibit a $p'$-order parabolic subgroup of strictly larger order, or at
least of larger order than that onf any $p'$-subgroup of $M$.

Let us note the following general result for certain primes:

\begin{prop}   \label{prop:not parabolic}
 Let $G=G(q)$ be simple of Lie type, not a Suzuki or Ree group. Let $p|(q-1)$
 but prime to the order of the Weyl group of $G$. If $G$ has property $(I_p)$
 with respect to $H\le G$, then $H$ is $G$-irreducible, that is, it is
 not contained in any proper parabolic subgroup of $G$.
\end{prop}

\begin{proof}
Under our assumptions on $p$, by \cite[Prop.~8.11]{Ca18} the decomposition
matrix of the principal $p$-block of $G$ is lower triangular with the only
unipotent character involved in $\Phi_{1_G}$ being the principal character.
On the other hand,
all constituents of the permutation character of $G$ on a parabolic subgroup
$P$ are unipotent, so if $P$ is proper, $\Ind_P^G(1_P)$ contains non-trivial
unipotent characters. Hence $G$ cannot have property $(I_p)$ with respect to
any subgroup of a proper parabolic subgroup.
\end{proof}

\begin{exmp}
 The conclusion of Proposition~\ref{prop:not parabolic} can fail when $p$
 divides the order of the Weyl group. For example, $G=\PSL_3(4)$ has property
 $(I_3)$ with respect to a subgroup $2^4.D_5$, and the latter is contained in a
 proper parabolic subgroup by Borel--Tits \cite[Thm~26.5]{MT}.
\end{exmp}

\subsection{The linear groups}

Our induction base is the following result, which also covers the defining
characteristic (leading to items~(4)--(6)):

\begin{prop}   \label{prop:L2}
 Let $G=\PSL_2(q)$, $q\ge7$. Then $G$ satisfies $(I_p)$ for $p$ dividing $|G|$
 if and only if one of
 \begin{enumerate}
  \item[\rm(1)] $2<p|(q+1)$, $H=B$, $\dim\Phi_{1_G}=q+1$;
  \item[\rm(2)] $q=2^f$, $p=2^f-1$ a Mersenne prime, $H=D_{2(q+1)}$, 
   $\dim\Phi_{1_G}=q(q-1)/2$;
  \item[\rm(3)] $p=2$, $q$ is odd, $H=O^2(B)$, $\dim\Phi_{1_G}=(q-1)_2(q+1)/2$;
  \item[\rm(4)] $G=\PSL_2(7)$, $H=\fS_4$, $p=7=\dim\Phi_{1_G}$;
  \item[\rm(5)] $G=\PSL_2(11)$, $H=\fA_5$, $p=11=\dim\Phi_{1_G}$; or
  \item[\rm(6)] $p=2$, $q$ is even, $H=C_{q+1}$, $\dim\Phi_{1_G}=q(q-1)$,
 \end{enumerate}
 where $B<G$ is a Borel subgroup.
\end{prop}

\begin{proof}
First assume that $p$ is odd. If $p|(q+1)$ the 2-transitive permutation
representation on a Borel subgroup gives~(1) by Lemma~\ref{lem:2-trans}. If
$p|(q-1)$ the Brauer tree in \cite{Bu76} shows $\dim\Phi_{1_G}=(q-1)u(q+1-u)/2$,
where $u=(q-1)_{p'}$. This divides $|G|$ only when $u=1$, so $q$ is even and
$q-1$ is a $p$-power. Thus, by Catalan's conjecture $q$ is a Mersenne prime as
in~(2). For $p|q$ we
have $\dim\Phi_{1_G}=(2^f-1)q$, where $q=p^f$, see \cite[Hauptsatz~9.4]{Bu76}.
The $p'$-subgroups $H$ of $G$ of largest order, the normalisers of non-split
maximal tori,
have order~$q+1$, or $q<60$ and $H$ is one of $\fA_4$, $\fS_4$ or $\fA_5$.
In the first case, $|G:H|$ is too large, while the last three cases can be
checked to only lead to~(4) and~(5).   \par
Now assume $p=2$. If $q\equiv3\pmod4$ then a Borel subgroup $B$ of order
$q(q-1)/2$ has odd order, and the permutation character for the 2-transitive
action on its cosets is $\Phi_{1_G}$. If $q\equiv1\pmod4$, the decomposition
matrix in \cite[VIII(a)]{Bu76} shows that
$\dim\Phi_{1_G}=(q+1)(q-1)_2/2=|G:O^2(B)|$ and we obtain~(3). Part~(6)
follows from the decomposition numbers in \cite[Haupts\"atze~7.5 and~7.9]{Bu76}.
\end{proof}

We next extract the necessary information on large subgroups from \cite{Lie85}.

\begin{prop}   \label{prop:large SLn}
 Let $M$ be a maximal subgroup of $\PSL_n(q)$, $n\ge3$, of order at least
 $q^{n(n-1)/2}$. Then $M$ is the image in $\PSL_n(q)$ of the intersection with
 $\SL_n(q)$ of one of the following subgroups of $\GL_n(q)$:
 \begin{enumerate}
  \item[\rm(1)] a maximal parabolic subgroup;
  \item[\rm(2)] $\GL_{n/2}(q)\wr\fS_2$, $\GL_{n/2}(q^2).2$, or $\Sp_n(q)$ when
   $n$ is even;
  \item[\rm(3)] $\GL_n(\sqrt q).2$ or $\GU_n(\sqrt q).2$ if $q$ is a square;
  \item[\rm(4)] $C_7.3$ in $\PSL_3(2)$; $C_{13}.3$ in $\PSL_3(3)$;
   $\fA_6$ or $3^2.Q_8$ in $\PSL_3(4)$; or $\fA_7$ in $\PSL_4(2)$.
 \end{enumerate}
\end{prop}

\begin{proof}
The main result of \cite{Lie85} characterises the maximal subgroups of
$\PSL_n(q)$ of order at least~$q^{3n}$. Using \cite[Tab.~3.5.A]{KL} and the
known order formulas one arrives at the cases in (1)--(3). For $n\ge7$ we have
$q^{3n}\le q^{n(n-1)/2}$, so no further examples arise. For $n\le6$ the
additional groups can be read off from \cite[Tab.~8.3--8.25]{BHR}.
\end{proof}

In what follows, we write $e_p(q)$ for the multiplicative order of $q$
in the finite field~$\FF_p$.

\begin{thm}   \label{thm:psl}
 Let $G=\PSL_n(q)$ with $n\ge3$ and $p$ a prime dividing $|G|_{q'}$. Then $G$
 satisfies $(I_p)$ with respect to some $p'$-subgroup $H$ if and only if one of
 \begin{enumerate}
  \item[\rm(1)] $e_p(q)=n$, $H=q^{n-1}.\GL_{n-1}(q)/C_d$ and
   $\dim\Phi_{1_G}= (q^n-1)/(q-1)$, where $d=\gcd(n,q-1)$; or
  \item[\rm(2)] $G=\PSL_3(4)$, $p=3$, $H=2^4.D_5$ and $\dim\Phi_{1_G}= 126$.
 \end{enumerate}
\end{thm}

\begin{proof}
We argue by induction over $n$. Let $e:=e_p(q)$. If $e>n$ then $|G|$ is prime
to~$p$. If $e=n$ then the end-node parabolic subgroups of $G$ have order
coprime to $p$ and the action on the cosets is 2-transitive. This gives
conclusion~(1) by Lemma~\ref{lem:2-trans}. For $2e>n$ Sylow $p$-subgroups of
$G$ are cyclic. Since the trivial character is not connected to the
exceptional character on the Brauer tree, we don't get examples for $e<n<2e$
by Lemma~\ref{lem:cyc}, using that $G$ has no such 2-transitive
permutation representations by \cite{C81}.
\par
Now assume that $e\le n/2$. Note that a Sylow $r$-subgroup, for $r|q$ the
defining characteristic of $G$, has order $q^{n(n-1)/2}$. Thus, by
Lemma~\ref{lem:order} a candidate
$p'$-subgroup $H\le G$ with respect to which $G$ satisfies $(I_p)$ must lie in
one of the subgroups listed in Proposition~\ref{prop:large SLn}. First assume
$H\le P$, with $P$ a maximal parabolic subgroup. Then a Levi factor $L$ of
$P$ is the quotient of $P$ by its maximal normal unipotent subgroup, and the
derived subgroup of $L$ modulo the centre has the form
$\bar L=\PSL_{n-k}(q)\times\PSL_k(q)$ for some $1\le k\le n/2$.
If $P$ contains an admissible~$H$, then $\bar L$ has property $(I_p)$, by
Lemma~\ref{lem:normal}. Thus, by induction, $k\le n-k\le e$ and so
$k=e$, $n=2e$, or we have $e=1$ and either we are in case~(2) or~(3) of
Proposition~\ref{prop:L2}, whence $k,n-k\le2$, or $p=3$, $q=4$ and $k,n-k\le3$.
We postpone the latter case for the moment. Now for $k=e$, $n=2e$
the permutation character of $G$ on $P$ contains exactly those unipotent
characters labelled by the constituents of the permutation character of
$\fS_e\times\fS_e$ inside $\fS_{2e}$, by Howlett--Lehrer theory. Thus it
contains the unipotent character $\rho$ labelled by the partition $(2e-1,1)$.
But that has only one $e$-hook, so $\rho$ does not lie in the principal block
by \cite{FS82} and we do not get an example when $n=2e$.
\par
Let us now discuss the exceptional cases with $e=1$. If $p=q-1\ge7$ is a
Mersenne prime, with $q=2^f$, and $n\le4$ then $H$ cannot lie in a parabolic
subgroup by Proposition~\ref{prop:not parabolic}. Assume then that $p=2$,
and $n\le4$. The permutation character of $G=\PSL_3(q)$ on its maximal
parabolic subgroups contains the unipotent character labelled $(2,1)$, while
this is not a constituent of~$\Phi_{1_G}$, by \cite[p.~253]{Ja90}. Also, the
permutation character of $G=\PSL_4(q)$ on a maximal parabolic subgroup of
type $\GL_2(q)^2$ contains the character labelled $(2,2)$, which does not
appear in $\Phi_{1_G}$ by \cite{Ja90} again. Finally, if $p=3$, $q=4$, then
$n\le6$. By \GAP\ we do not
get an example in $\PSL_4(4)$ with $p=3$. For $\PSL_5(4)$ and $\PSL_6(4)$ the
permutation characters on the relevant parabolic subgroups involve the
unipotent characters labelled $(4,1)$, $(4,2)$ respectively, neither of which
occurs in $\Phi_{1_G}$, by \cite[p.~258/259]{Ja90}. Hence these exceptional
cases do not propagate to further examples.
\par
This completes the discussion of the maximal parabolic subgroups. Now assume
that $n$ is even and $H$ is contained in (the image in $G$ of) one of
$M=\GL_{n/2}(q)\wr\fS_2$, $\GL_{n/2}(q^2).2$ or $\Sp_n(q)$. Since $e\le n/2$,
the order of $M$ is divisible by $p$. Induction shows that we must have
$e=n/2$, again. But any $p'$-subgroup of~$M$ has index at least
$(q^{n/2}-1)^2/(q-1)^2$, and thus order less than $q^{n(n-1)/2}$. Then it
cannot lead to an example by Lemma~\ref{lem:order}. The same argument
applies to the subfield subgroups in Proposition~\ref{prop:large SLn}(3).
For the groups $\PSL_n(2)$, $n\le4$, and $\PSL_3(4)$ all primes can be checked
using \GAP\ and only the case $p=3$ for $\PSL_3(4)$ arises.
\end{proof}

For later inductive purposes, let's point out the following consequence:

\begin{cor}   \label{cor:psl}
 Assume $\PSL_n(q)$, $n\ge2$, has property $(I_p)$, for $p$ dividing
 $|\PSL_n(q)|_{q'}$, and set
 $e:=e_p(q)$. Then either $n\le e$, or $e=1$ and one of $n=2$ or $n=p=q-1=3$.
\end{cor}

\subsection{The unitary groups}
For the base case of the unitary series we can again also treat the defining
characteristic case (which does not lead to examples).
 
\begin{prop}   \label{prop:U3}
 Let $G=\PSU_3(q)$, $q\ge3$. Then $G$ satisfies $(I_p)$ if and only if one of
 \begin{enumerate}
  \item[\rm(1)] $p|(q+1)$, $H=O^p(B)$ where $B$ is a Borel subgroup of $G$; or
  \item[\rm(2)] $p|(q^3+1)$ but $p{\not|}(q+1)$, $H$ is a maximal parabolic
   subgroup, $\dim\Phi_{1_G}=q^3+1$.
 \end{enumerate}
\end{prop}

\begin{proof}
For $p$ dividing $q+1$ this follows from the decomposition matrix given in
\cite[Thm~4.3]{Ge90} when $p\ne 2,3$ and from \cite[Thm~4.5]{Ge90} when $p=3$.
When $p=2$ we use that the decomposition matrix of the Weyl group $C_2$ embeds
into that of $G$, so the 1-PIM does contain the trivial and the Steinberg
character once each, as does the permutation character on $O^2(B)$. Since the
unipotent characters form a basic set, the latter must indeed be indecomposable.
Now assume that $p|(q^3+1)$ but not $q+1$, so in particular
$p\ge5$. Then Sylow $p$-subgroups of $G$ are cyclic, and the permutation action
on the cosets of a maximal parabolic subgroup is 2-transitive. We conclude
using the Brauer tree from \cite[Thm~4.2]{Ge90}.   \par
Next assume that $p|(q-1)$ but $p\ne2$. Again Sylow $p$-subgroups of $G$ are
cyclic in this case, and by \cite[Thm~4.1]{Ge90} we have
$$\dim\Phi_{1_G}=(q-1)(q^3+1-u(q^2+q+1))/(2u)<q^4/6,$$
where $u=(q-1)_{p'}$. Now $H$ cannot be contained in a parabolic subgroup,
by Proposition~\ref{prop:not parabolic}. But there are no subgroups of order
at least $|G|/\dim\Phi_{1_G}$ that are irreducible, see
\cite[Tab.~8.5 and~8.6]{BHR}.   \par
Finally, assume $p|q$. For $q\le 11$ we may use \cite{GAP} to verify our claim.
For $q>11$ the decomposition matrix is not known at present. But for those $q$,
the $p'$-subgroups of largest order are contained in the
normalisers $H$ of maximal tori, of order $6(q+1)^2/\gcd(3,q+1)$, see
\cite[Tab.~8.5 and~8.6]{BHR}. Direct calculation shows that the restriction of
the Steinberg character of $G$ to $H$ contains the trivial character. Since the
Steinberg character is of $p$-defect zero and hence does not lie in the
principal $p$-block of $G$, this shows that $H$ does not lead to an example.
\end{proof}

As in the linear case we now determine relevant subgroups.

\begin{prop}   \label{prop:large SUn}
 Let $M$ be a maximal subgroup of $\PSU_n(q)$, $n\ge4$, of order at least
 $q^{n(n-1)/2}$. Then $M$ is the image in $\PSU_n(q)$ of the intersection with
 $\SU_n(q)$ of one of the following subgroups of $\GU_n(q)$:
 \begin{enumerate}
  \item[\rm(1)] a maximal parabolic subgroup;
  \item[\rm(2)] $\GU_d(q)\times\GU_{n-d}(q)$ with $1\le d<n/2$;
  \item[\rm(3)] $\GU_{n/2}(q)\wr\fS_2$, $\GL_{n/2}(q^2).2$, or $\Sp_n(q)$ when
   $n$ is even;
  \item[\rm(4)] $\GU_n(\sqrt q).2$ if $q$ is a square;
  \item[\rm(5)] $\SO_n(q)$ when $q,n$ are odd;
  \item[\rm(6)] $\SO_n^\pm(q)$ when $q$ is odd, $n$ is even;
  \item[\rm(7)] $3^3.\fS_4$ in $\PSU_4(2)$; $\fA_7$, $\PSL_3(4).2^2$ or
   $2^4.\fA_6$ in $\PSU_4(3)$; $3^4.\fS_5$ in $\PSU_5(2)$; or $3.M_{22}$ or
   $3_1.\PSU_4(3).2$ in $\PSU_6(2)$.
 \end{enumerate}
\end{prop}

\begin{proof}
This again follows from \cite{Lie85} using \cite[Tab.~3.5.B]{KL} and \cite{BHR}.
\end{proof}

We now show that generically, there are no examples for unitary groups.
Via Harish-Chandra theory the principal series unipotent characters of unitary
groups $\SU_n(q)$ are in bijection with characters of its Weyl group of type
$B_k$, $k=\lfloor n/2\rfloor$, see \cite[Prop.~4.3.6]{GM20}. Recall that the
irreducible characters of the Weyl group of type $B_k$ are labelled by
bi-partitions of $k$, see e.g.~\cite[5.5.4]{GP}. The following, which is shown
in
\cite[Prop.~6.4.7]{GP}, will be used to determine the constituents of the
permutation character on maximal parabolic subgroups:

\begin{lem}   \label{lem:const Sp}
 Let $W$ be the Weyl group of type $B_n$ and $W_d$ its standard parabolic
 subgroup of type $B_{n-d}\times A_{d-1}$, with $1\le d\le n$. Then the
 constituents of the permutation character of $W$ on $W_d$ are the irreducible
 characters labelled by bi-partitions $(n-d+k,l;d-k-l)$ with $0\le k,l$ and
 $k+l\le d$.
\end{lem}

\begin{thm}   \label{thm:PSU}
 Let $G=\PSU_n(q)$ with $n\ge4$ and $p$ a prime dividing $|G|_{q'}$.
 If $G$ satisfies $(I_p)$ then $p=2$ or $(p,q)=(3,2)$.
\end{thm}

\begin{proof}
Throughout we may and will assume $p>2$ and $(p,q)\ne(3,2)$. Set $e:=e_p(q)$ and
$e':=2e$ when $e$ is odd, $e':=e/2$ when $e\equiv2\pmod 4$, and $e':=e$ when
$e\equiv0\pmod4$. The Sylow $p$-subgroups of $G$ are cyclic when $e'>n/2$, and
in this case no examples with $p$ dividing $|G|$ can arise by
Lemma~\ref{lem:cyc} combined with \cite{C81} and the Brauer trees in
\cite{FS90}. So assume $e'\le n/2$. Let $H<G$ be a $p'$-subgroup such that
$\Phi_{1_G}=\Ind_G^H(1_H)$. Since a Sylow $r$-subgroup of $G$, for $r|q$, has
order $q^{n(n-1)/2}$, $H$ will lie in one of the maximal subgroups $M$ listed
in Proposition~\ref{prop:large SUn}. First assume $M$ is a maximal parabolic
subgroup. Then its Levi subgroups have a subquotient
$\PSU_{n-2d}(q)\times\PSL_d(q^2)$ for some $1\le d\le n/2$. By our results on
the linear groups (Corollary~\ref{cor:psl}) we must have $d\le e$; the case
$e=1$, $d=2$ cannot occur because $\GL_2(q^2)$ is not an example for $e=1$ by
Proposition~\ref{prop:L2}. 
\par
First assume $e$ is odd. Then $n-2d<e'=2e$ by induction, so
$2e=e'\le n/2 <e+d\le 2e$, which is not possible. If $e=2e'$ is twice an odd
number, again by induction we have $d\le e/2$ and $n-2d<e'$, whence
$2e'\le n< 3e'$ and $d>(n-e')/2\ge e'/2$. In this case the permutation
character on $P$ contains the unipotent character labelled by the bi-partition
$\mu=(\lfloor(n-2)/2\rfloor:1)$, hence by the partition
$$\la=\begin{cases} (n-2,2)& \text{if $n$ is even,}\\
  (n-2,1^2)& \text{if $n$ is odd,}\end{cases}$$
with $2$-quotient $\mu$.
Since the $e'$-core of $\la$ is not equal to the $e'$-core of $(n)$, this
unipotent character does not lie in the principal $p$-block of $G$ by
\cite{FS89}, so $H$ cannot be contained in $P$.
\par
Finally, if $e$ is divisible by~4, then our assumptions and induction force
$n-2d<e$ and $d\le e/2$, so $e\le n/2 <(2d+e)/2\le e$, hence no case arises.
This completes the discussion of maximal parabolic subgroups.
\par
The groups listed in Proposition~\ref{prop:large SUn}(2)--(6) have order
divisible by $p$, and their largest $p'$-subgroups have smaller order than
a Sylow $r$-subgroup of $G$. Finally, the cases in
Proposition~\ref{prop:large SUn}(7) do not lead to examples by \GAP.
\end{proof}

We conclude our discussion of the unitary groups by dealing with the two
cases left open in the previous result.

\begin{prop}    \label{prop:Un(2)}
 Let $G=\PSU_n(2)$ with $n\ge4$. Then $G$ satisfies $(I_3)$ if and only if
 $n\le7$. Here, $H=[2^{(n^2-n-4)/2}].D_5$
\end{prop}

\begin{proof}
The claim for $n=4,5,6$ can be checked using \GAP. Now let $G=\PSU_7(2)$. The
Harish-Chandra induction $\Psi$ of the two 1-PIMs of the parabolic subgroups of
types $\GU_5(2)$ and $\GL_3(4)$ to $G$ have the same ordinary constituents and
hence agree. We claim they are indecomposable and thus $G$ satisfies $(I_3)$.
By Lemma~\ref{lem:normal} it is sufficient to show the analogous statement for
$\GU_7(2)$. Here, by \cite[Thm~(8A)]{FS82} the unipotent characters
form a basic set for the unipotent blocks. The Harish-Chandra restriction of
any proper non-zero subcharacter of $\Psi$ to the two types of maximal parabolic
subgroups does not decompose as a non-negative integral linear combination of
projectives, so indeed $\Psi$ is indecomposable.   \par
Now assume $n\ge8$. We argue that a putative $p'$-subgroup $H$ cannot be
contained inside any of the subgroups listed in
Proposition~\ref{prop:large SUn}. Let first $P$ be a parabolic subgroup, of
type $\GU_{n-2d}(2)\GL_d(4)$, for some $1\le d\le n/2$. By induction and
Theorem~\ref{thm:psl} we then have $d\le3$ and $n-2d\le7$, so $n\le13$. By
Lemma~\ref{lem:const Sp}, for
$n=8$ and $n=9$ the permutation character on $P$ contains the unipotent
character labelled by the bi-partition $(31;-)$ which is not a constituent of
the induced 1-PIM from a parabolic subgroup of type $\GL_4(4)$,
for $n=10,11$ it contains $(41;-)$ which is not in the 1-PIM induced from
$\GL_5(4)$, and for $n=12,13$ the character labelled $(42;-)$ has this property.
\par
For the non-parabolic subgroups in Proposition~\ref{prop:large SUn} we can
argue as in the generic case in Theorem~\ref{thm:PSU}.
\end{proof}

\begin{prop}    \label{prop:Un p=2}
 Let $G=\PSU_n(q)$ with $n\ge4$. Then $G$ satisfies $(I_2)$ if and only if
 $n\le5$. Here, $H=O^2(B)$, for $B$ a Borel subgroup of $G$.
\end{prop}

\begin{proof}
First let $n=4$. Here the Levi factors of both maximal parabolic subgroups have
property $(I_2)$ by Proposition~\ref{prop:L2}, as does the subgroup
$\PSp_4(q)$. In view of Corollary~\ref{cor:compfac}, to show that $G$ has
property $(I_2)$ it suffices to prove this for $\tilde G:=\GU_4(q)$. Now the
1-PIM of the Levi factor $\GL_2(q^2)$ contains both unipotent characters
once, and thus its Harish-Chandra induction $\Psi$ to $\tilde G$ contains all
unipotent characters, with the multiplicity given by the character degrees in
the Weyl group $W(B_2)$ \cite[Thm~3.2.27]{GM20}. On the other hand, the
decomposition matrix of the Weyl group embeds into the decomposition matrix of
$\tilde G$; since $W(B_2)$ is a 2-group, its 1-PIM is the regular
representation.
This means that the 1-PIM of $\tilde G$ contains exactly the same unipotent
constituents as $\Psi$. Since Harish-Chandra induction sends unipotent
blocks to unipotent blocks \cite[Prop.~3.3.20]{GM20}, $\tilde G$ has a unique
unipotent 2-block \cite{En00}, and the unipotent characters form a basic set
for the unipotent blocks \cite{FS82}, this
shows that indeed $\Psi=\Phi_{1_{\tilde G}}$, so $\tilde G$ satisfies $(I_2)$.
\par
For $n=5$ the exactly same argument as for $n=4$ shows that again $H=O^2(B)$ is
an admissible subgroup. For $n\ge6$ first consider maximal parabolic subgroups.
For $P$ of type $\GU_{n-2d}(q)\GL_d(q^2)$ we have, by induction, $n-2d\le5$
and $d\le2$, so $n\le9$. Here, for $n=6,7$ the unipotent constituent
labelled by the bi-partition $(21;-)$ of the permutation
character on $P$ does not occur in the induced 1-PIM of the parabolic subgroup
of type $\GL_3(q^2)$, and for $n=8,9$, it is the constituent $(2^2;-)$ which
is not in the induced 1-PIM from $\GL_4(q^2)$.
\par
This deals with parabolic overgroups of a possible subgroup $H$. The arguments
for the non-parabolic maximal subgroups in Proposition~\ref{prop:large SUn} are
as in the generic case.
\end{proof}

\subsection{The symplectic groups}
We next discuss the induction base for symplectic groups.

\begin{prop}   \label{prop:Sp4}
 Let $G=\PSp_4(q)$, $q\ge3$, and $p$ a prime dividing $|G|_{q'}$. Then $G$
 satisfies $(I_p)$ if and only if $p=2$, $H=O^2(B)$,
 $\dim\Phi_{1_G}=(q-1)_2^2(q+1)^2(q^2+1)$, where $B$ is a Borel subgroup of $G$.
\end{prop}

\begin{proof}
It follows from the decomposition matrix given in \cite[Thm~3.1]{Wh90} that we
get the stated example for $p=2$. Now assume that $p$ is odd. For $p|(q^2+1)$
the
Sylow $p$-subgroups are cyclic, so we can conclude with Lemma~\ref{lem:cyc},
using that $\PSp_4(q)$ does not have a corresponding 2-transitive permutation
representation for $q>2$, by \cite{C81}. If $p|(q+1)$, the decomposition matrix
in \cite[Thm~4.2]{Wh90b} shows that $\dim\Phi_{1_G}=(q+1)(q^3+1)$, which does
not divide $|G|$. Finally, when $p|(q-1)$, then  a putative $p'$-subgroup $H$
cannot lie inside a parabolic subgroup, by
Proposition~\ref{prop:not parabolic}.
Any other subgroup, by \cite[Tab.~8.12 and~8.13]{BHR} has $p'$-subgroups of
order smaller than $q^4$, and so cannot contain a suitable $H$.
\end{proof}

\begin{prop}   \label{prop:large Sp_n}
 Let $M$ be a maximal subgroup of $\PSp_{2n}(q)$, $n\ge3$, of order at least
 $q^{n^2}$. Then $M$ is the image in $\PGL_{2n}(q)$ of one of the following
 subgroups of $\Sp_{2n}(q)$:
 \begin{enumerate}
  \item[\rm(1)] a maximal parabolic subgroup;
  \item[\rm(2)] $\Sp_{2d}(q)\times\Sp_{2n-2d}(q)$ with $1\le d< n/2$;
  \item[\rm(3)] $\Sp_n(q)\wr\fS_2$ or $\Sp_n(q^2).2$ when $n$ is even
  \item[\rm(4)] $\GL_n(q).2$ or $\GU_n(q).2$;
  \item[\rm(5)] $\Sp_{2n}(\sqrt{q})$ if $q$ is a square;
  \item[\rm(6)] $\GO_{2n}^\pm(q)$ if $q$ is even;
  \item[\rm(7)] $\Sp_2(q^3).3$ or $\Sp_2(q)\wr\fS_3$ when $n=3$;
  \item[\rm(8)] $G_2(q)$ when $n=3$ and $q$ is even;
  \item[\rm(9)] $2.J_2$ in $\PSp_6(5)$ or $\PSp_6(9)$; $\fS_{10}$ in
   $\PSp_8(2)$; or $\fS_{14}$ in $\PSp_{12}(2)$.
 \end{enumerate}
\end{prop}

\begin{proof}
The maximal subgroups of $\Sp_{2n}(q)$ of order at least~$q^{6n}$ are described
in \cite{Lie85}. More details about these can be found in
\cite[Tab.~3.5.C]{KL}, and we arrive at~(1)--(7). For $n\ge6$ we always have
$q^{n^2}\ge q^{6n}$. For $n\le5$ the relevant subgroups can be read off
from \cite[Tab.~8.28--8.65]{BHR}.
\end{proof}

In the proof of the following result we make use of the description of
blocks and Brauer trees in classical groups by Fong and Srinivasan
\cite{FS89,FS90}.
 
\begin{thm}   \label{thm:PSp}
 Let $G=\PSp_{2n}(q)$ with $n\ge3$ and $p$ a prime dividing $|G|_{q'}$. Then $G$
 satisfies $(I_p)$ with respect to some $p'$-subgroup $H$ if and only if one of
 \begin{enumerate}
  \item[\rm(1)] $q=2$, $e_p(2)=2n$, $H=\GO_{2n}^+(2)$,
   $\dim\Phi_{1_G}=2^{n-1}(2^n+1)$; or
  \item[\rm(2)] $q=2$, $e_p(2)=n$ is odd, $H=\GO_{2n}^-(2)$,
   $\dim\Phi_{1_G}=2^{n-1}(2^n-1)$.
 \end{enumerate}
\end{thm}

\begin{proof}
We argue by induction over $n$. Let $e:=e_p(q)$. We set $e':=2e$ is $e$ is odd,
and $e':=e$ otherwise. The order of $G$ is coprime to $p$ if $e'>2n$. The Sylow
$p$-subgroups of $G$ are cyclic if $n<e'\le 2n$. The only 2-transitive actions
of $G$ are, by \cite{C81}, the ones given in~(1) and~(2) with $q=2$, and they
lead to examples by Lemma~\ref{lem:2-trans}. As the trivial character is never
connected to the exceptional node on the corresponding Brauer tree \cite{FS90},
there are no further examples in the cyclic Sylow case by Lemma~\ref{lem:cyc}.
\par
Now assume that $e'\le n$. We discuss the various maximal overgroups $M$ of
$p'$-subgroups $H$ of $G$. Note that a maximal unipotent subgroup of $G$ has
size $q^{n^2}$ so it suffices to consider the groups occurring in
Proposition~\ref{prop:large Sp_n}. Let $P$ be a maximal parabolic subgroup of
$G$. Then $G$ has a Levi factor with subquotient
$\bar L=\PSp_{2n-2d}(q)\times\PSL_d(q)$, for some $1\le d\le n$. By
Corollary~\ref{cor:compfac}, both factors must satisfy $(I_p)$. Hence, by
induction, Corollary~\ref{cor:psl} applied to
the $\PSL_d(q)$-factor imposes that $d\le e$, or $e=1$ and either $d=2$, or
$d=p=3$, $q=4$, and similarly, by Proposition~\ref{prop:Sp4}, the
$\PSp_{2n-2d}(q)$-factor forces one of
\begin{itemize}
\item $2n-2d<e'$,
\item $2n-2d=e'$, $q=2$,
\item $2n-2d=2$, or
\item $2n-2d=4$, $p=2$.
\end{itemize}
First assume $2(n-d)\le e'$. The constituents of
the permutation character on $P$ are described in Lemma~\ref{lem:const Sp}. For
$d\le e$, if $e$ and $n$ are both odd, the constituent labelled by the
bi-partition
$((n-2+e)/2,(n-e)/2;1)$ does not lie in the principal $p$-block, if $e$ is
odd and $n$ is even, the one labelled $((n-1+e)/2,(n+1-e)/2;-)$ is not in the
principal block; if $e$ is even, then $e\le n$, and the constituent labelled
$(n-e/2;e/2)$ is outside the principal block (\cite{FS89}). On the other
hand, if
$d=p=3$, $q=4$, then $G=\PSp_6(4)$, and $P$ has type $\GL_3(4)$. Here,
Harish-Chandra induction of the trivial character contains the unipotent
character labelled $(3;-)$, while this is not a constituent of the 1-PIM by
\cite[Thm~2.1]{Wh00}.

Next, assume $2n-2d=e'$, $q=2$ and $d=2$, $e=1$. Then $G=\PSp_6(2)$, which
can be discarded using \GAP.

Now consider the case that $2n-2d=2$, so $n=d+1$. If $d\le e$, then
$e'\le n=d+1\le e+1$, so $e=e'$ is even and equal to $n-1$ or $n$. Here the
permutation character on $P$ of type $\Sp_2(q)\GL_{n-1}(q)$  contains
the unipotent characters labelled $(n-1;1)$ and $(n-3;3)$ which do not
lie in the principal $p$-block when $e=n$, $e=n-1$ respectively.
If $d=2$, $e=1$, then by Proposition~\ref{prop:not parabolic} this can only
give rise to an example if $p\le3$. For $p=3$ the Harish-Chandra induction of
the trivial character from $P$ of type $\Sp_2(q)\GL_2(q)$ contains the
unipotent character labelled $(3;-)$, while this is not a constituent of the
1-PIM by \cite[Thm~2.1]{Wh00}. When $p=2$ the 1-PIM of each simple factor of
$\bar L$
contains the Steinberg character, and so its Harish-Chandra induction to $G$
contains the unipotent character labelled $(-;2,1)$ not lying in the principal
block.
Next, if $d=3$, $p=3$, $q=4$, then $G=\PSp_8(4)$. Here, the Harish-Chandra
induced 1-PIM from a maximal parabolic subgroup $P$ of type $\Sp_6(q)$ (given
in \cite[Thm~2.1]{Wh00}) does not contain the unipotent character labelled
$(3;1)$, but the permutation character on $P$ does, so $H$ cannot be contained
in~$P$.

Finally, consider $2n-2d=4$, $p=2$ and $e=1$, so $n=d+2$. If $d\le e=1$ then
$n=3$ and $G=\PSp_6(q)$. Here the induced 1-PIM from the parabolic subgroup
$P$ of type $\Sp_4(q)$ is not contained in the induced 1-PIM from the parabolic
subgroup of type $\GL_3(q)$, so $H$ cannot lie inside $P$. The last remaining
case is now when $d=2$, so $n=4$, $G=\PSp_8(q)$. Here again the induced 1-PIM
from the parabolic subgroup $P$ of type $\Sp_4(q)\GL_2(q)$ is not
contained in that from the parabolic subgroup of type $\GL_4(q)$.
This completes the discussion of maximal parabolic subgroups.
\par
Since $e'\le n$, the groups in (2)--(8) of Proposition~\ref{prop:large Sp_n}
have order divisible by~$p$. Their largest $p'$-subgroups are of order less
than $q^{n^2}$, so they cannot contain a relevant subgroup~$H$. 
\par
As far as the groups in (9) of Proposition~\ref{prop:large Sp_n} are
concerned, $\fS_{10}$ and $\fS_{14}$ do not satisfy $(I_p)$ for any prime
divisor of their order by Theorem~\ref{thm:alt}, and the Sylow $p$-subgroups of
$G$ are cyclic for all other primes. The group $2.J_2$ has property $(I_p)$
only for $p=5$, by Theorem~\ref{thm:spor}, so only $G=\PSp_6(9)$ needs to be
considered. But here a
Borel subgroup of $G$ has larger order than any $5'$-subgroup of~$2.J_2$.
\end{proof}

\subsection{The orthogonal groups}

\begin{prop}   \label{prop:large SO_2n+1}
 Let $M$ be a maximal subgroup of $\OO_{2n+1}(q)$, $n\ge3$ and $q$ odd, of
 order at least $q^{n^2}$. Then $M$ is the intersection with $\OO_{2n+1}(q)$
 of one of the following subgroups of $\GO_{2n+1}(q)$:
 \begin{enumerate}
  \item[\rm(1)] a maximal parabolic subgroup;
  \item[\rm(2)] $\SO_{2d+1}(q)\times\GO_{2n-2d}^\pm(q)$ with $0\le d< n$;
  \item[\rm(3)] $\SO_{2n+1}(\sqrt{q}).2$ if $q$ is a square;
  \item[\rm(4)] $G_2(q)$ when $n=3$; or
  \item[\rm(5)] $2^6.\fA_7$, $\fS_9$ or $\PSp_6(2)$ in $\OO_7(3)$;
   or $2^8.\fA_9$ in $\OO_9(3)$.
 \end{enumerate}
\end{prop}

\begin{proof}
This follows again from \cite[Thm~4.2]{Lie85}, where maximal subgroups of
order at least $q^{4n+6}$ are identified, in conjunction with the explicit
descriptions of the generic subgroups in \cite[Tab.~3.5.D]{KL} and the complete
lists of maximal subgroups in \cite{BHR} for $n\le5$.
\end{proof}

\begin{thm}   \label{thm:PSO 2n+1}
 Let $G=\OO_{2n+1}(q)$ with $n\ge3$ and $q$ odd. Then $G$ does not satisfy
 $(I_p)$ for any prime $p$ dividing~$|G|_{q'}$.
\end{thm}

\begin{proof}
We again proceed by induction on $n$. The group $G$ does not have 2-transitive
permutation actions by \cite{C81}, so using \cite{FS90} no examples arise for
cyclic Sylow
$p$-subgroups. So setting $e=e_p(q)$ and $e'=2e/\gcd(2,e)$ we have $e'\le n/2$.
To see that a putative $p'$-subgroup $H$ can not lie inside a maximal parabolic
subgroup of $G$ we can argue as in the proof of Theorem~\ref{thm:PSp}, using
that groups of types $B_n$ and $C_n$ have the same Weyl group and hence the
same decomposition of Harish-Chandra induction.
\par
Again the remaining possibilities listed in Proposition~\ref{prop:large
SO_2n+1}, except for those in item~(5), have order divisible by $p$ and do not
contain large enough $p'$-subgroups. The group $\fA_9$ does not have property
$(I_p)$ for $p\le7$ by Theorem~\ref{thm:alt}, and Sylow $p$-subgroups of
$G=\OO_9(3)$ are cyclic for all larger primes. The case of $G=\OO_7(3)$ can be
discarded using the known decomposition matrices in \GAP.
\end{proof}

In the following we set $\GO_{2m+1}^0(q):=\GO_{2m+1}(q)$. The next two
results are proved in a very similar way to the earlier results of this
type, using \cite{Lie85}, \cite{KL} and \cite{BHR}:

\begin{prop}   \label{prop:large SO_2n+}
 Let $M$ be a maximal subgroup of $\OO_{2n}^+(q)$, $n\ge4$, of order at least
 $q^{n^2-n}$. Then $M$ is the intersection with $\OO_{2n}^+(q)$ of the image in
 $\PGL_{2n}(q)$ of one of the following subgroups of $\GO_{2n}^+(q)$:
 \begin{enumerate}
  \item[\rm(1)] a maximal parabolic subgroup;
  \item[\rm(2)] $\GO_d^\eps(q)\times\GO_{2n-d}^\eps(q)$ with $0<d<2n$,
   $\eps\in\{0,\pm\}$;
  \item[\rm(3)] $\Sp_{2n-2}(q)$ when $q$ is even;
  \item[\rm(4)] $\GO_n^\pm(q)\wr\fS_2$ or $\GO_n^+(q^2)$ when $n$ is even;
  \item[\rm(5)] $\GO_n(q)\wr\fS_2$ or $\GO_n(q^2)$ when $nq$ is odd;  \item[\rm(6)] $\GL_n(q).2$ or $\GU_n(q)$ when $n$ is even;
  \item[\rm(7)] $\GO_{2n}^\pm(\sqrt{q})$ when $q$ is a square; or
  \item[\rm(8)] $\fA_9$ in $\OO_8^+(2)$; $\OO_8^+(2)$ in $\OO_8^+(3)$; or
   $\fA_{16}$ in $\OO_{14}^+(2)$.
 \end{enumerate}
\end{prop}

\begin{prop}   \label{prop:large SO_2n-}
 Let $M$ be a maximal subgroup of $\OO_{2n}^-(q)$, $n\ge4$, of
 order at least $q^{n^2-n}$. Then $M$ is the intersection with $\OO_{2n}^-(q)$
 of the image in $\PGL_{2n}(q)$ of one of the following subgroups of
 $\GO_{2n}^-(q)$:
 \begin{enumerate}
  \item[\rm(1)] a maximal parabolic subgroup;
  \item[\rm(2)] $\GO_d^\eps(q)\times\GO_{2n-d}^{-\eps}(q)$ with $0<d<2n$,
   $\eps\in\{0,\pm\}$;
  \item[\rm(3)] $\Sp_{2n-2}(q)$ when $q$ is even;
  \item[\rm(4)] $\GO_n^-(q^2)$ when $n$ is even;
  \item[\rm(5)] $\GO_n(q)\wr\fS_2$ or $\GO_n(q^2)$ when $nq$ is odd;
  \item[\rm(6)] $\GU_n(q)$ when $n$ is odd; or
  \item[\rm(7)] $\fA_{12}$ in $\OO_{10}^-(2)$; or $\fA_{13}$ in $\OO_{12}^-(2)$.
 \end{enumerate}
\end{prop}

\begin{thm}   \label{thm:PSO 2n}
 Let $G=\OO_{2n}^\pm(q)$ with $n\ge4$. Then $G$ does not satisfy $(I_p)$ for
 any prime $p$ dividing~$|G|_{q'}$.
\end{thm}

\begin{proof}
We proceed as in our earlier proofs. Let $e:=e_p(q)$ and $e':=2e/\gcd(2,e)$.
If $e'>n$ then Sylow $p$-subgroups of $G$ are cyclic and we conclude with
Lemma~\ref{lem:cyc}, again using \cite{FS90}. Thus, $e'\le n$. Now first assume
$H$ lies in some
maximal parabolic subgroup $P$ of $G=\OO_{2n}^+(q)$. Then a Levi factor of
$P$ has a subquotient of the form $\bar L=\OO_{2n-2d}^+(q)\times\PSL_d(q)$.
By induction and Corollary~\ref{cor:psl} we have $2(n-d)<\max\{e',5\}$ and
$d\le e$ or $d\le3,e=1$.
First assume $e$ is odd, so $e'=2e$. Then our inequalities force $e=1$, and
either $d=2$, $n=4$, or $d=p=3$, $q=4$ and $n\le5$. If $d=2$ we have $P$ is of
type $\SO_4^+(q)\GL_2(q)$. For $p\ge3$ the permutation character on~$P$
contains the unipotent characters labelled $(2^2;-)^\pm$, not contained in the
induced
1-PIM of the parabolic subgroup of type $\GL_4(q)$ (known from \cite{Ja90}),
while for $p=2$, the latter contains the character $(2^2;-)^\pm$ with a smaller
multiplicity than the former. If $d=p=3$, $P$ is of type
$\SO_{2n-6}^+(q)\GL_3(q)$. This is not a maximal parabolic subgroup when $n=4$.
For
$n=5$ the permutation character on $P$ contains the character labelled $(1;3,1)$
which is not a constituent of the 1-PIM induced from a parabolic subgroup of
type $\GL_5(q)$. This concludes the case when $e$ is odd.
\par
If $e$ is even, then $2(n-d)<\max\{e,5\}$ or $2(n-d)=6$ and $e=4$. If $e=4$
as $d\le e$, we have $n\le7$. Now the induced 1-PIM for $P$ of type
$\SO_6^+(q)\GL_d(q)$ contains the unipotent characters labelled $(3;1)$,
$(41;-)$, $(2^21^2;-)^\pm$, $(61;-)$ for $n=4,5,6,7$ respectively, not lying
in the principal $p$-block. For $P$ of type $\SO_4^+(q)\GL_d(q)$ it again
contains $(3;1)$, $(41;-)$, $(2^21^2;-)^\pm$ for $n=4,5,6$ respectively.
If $e=2$ then only $P$ of type $\SO_4^+(q)\GL_2(q)$ needs to be considered.
Here the character labelled $(2,1;1)$, not in the principal $p$-block, is a
constituent of the permutation character. Hence we have $e\ge6$ and
$n<3e/2\le 3n/2$. Here the unipotent character labelled $(n+1-e,e-1;-)$ is a
constituent of the permutation character on $P$, but does not lie in the
principal $p$-block.
\par
Turning to the groups of minus type, assume now $P$ is maximal parabolic in
$G=\OO_{2n}^-(q)$, of type $\SO_{2n-2d}^-(q)\GL_d(q)$. Again first assume
$e$ is odd. Then as before this forces $e=1$, and either $d=2$, $n=4$, or
$d=p=3$, $q=4$ and $n=5$. In the first case, the relevant maximal parabolic
subgroup $P$ has a subquotient $\OO_4^-(q)\PSL_2(q)=\PSL_2(q^2)\PSL_2(q)$, and
by
Proposition~\ref{prop:L2} we need to have $p=2$. In this case, the permutation
character on $P$ contains the unipotent character labelled $(2,1;-)$, but this
is not a constituent of the induced 1-PIM from a maximal parabolic subgroup of
type $\GL_3(q)$. When $d=p=3$, $q=4$ and $n=5$ then $P$ is of type
$\OO_4^-(4)\PSL_3(4)=\PSL_2(16)\PSL_3(4)$, but this does not have property
$(I_3)$ by Proposition~\ref{prop:L2}.   \par
So now assume $e$ is even. Then our conditions yield $2(n-d)<\max\{e,5\}$.
Since Sylow $p$-subgroups of $G$ are still cyclic when $n=e$ we may also
assume $n\ge e+1$. If $e=2$ then $P$ has type $\SO_4^-(q)\GL_2(q)$, and again
the permutation character on $P$ contains the character $(2,1;-)$, which is not
a constituent of the induced 1-PIM from a parabolic subgroup of type $\GL_3(q)$.
If $n\ge4$ then furthermore, $n\le d+e/2\le 3e/2$. In this case, by
Lemma~\ref{lem:const Sp} the permutation character on $P$ contains the
principal series unipotent character labelled $(e-1;n-e)$ which does not lie
in the principal $p$-block of $G$.
\par
The groups $M$ in Proposition~\ref{prop:large SO_2n+}(2)--(7) and in
Proposition~\ref{prop:large SO_2n-}(2)--(6) have order divisible by $p$, and
thus the order of a maximal $p'$-subgroup of $M$ is smaller than $q^{n^2-n}$,
the size
of a Sylow $r$-subgroup of $G$, where $r|q$. Finally, the groups $\fA_9$,
$\OO_8^+(2)$, $\fA_{12}$ and $\fA_{13}$ do not have property $(I_p)$ for
$p\le7$, $p\le7$, $p\le 11$, $p\le11$ respectively, and the Sylow $p$-subgroups
of $\OO_8^+(2)$, $\OO_8^+(3)$, $\OO_{10}^-(2)$ and $\OO_{12}^-(2)$,
respectively, are cyclic for all larger primes, so no examples can arise from
these.
\end{proof}

\subsection{Groups of exceptional type}
We now discuss the exceptional groups of Lie type. Structural results on
their Sylow $p$-subgroups are given in \cite[Thm~25.14]{MT}, the characters in
the principal blocks are described in \cite{BMM} and \cite{En00}. We use
\Chevie\ \cite{Mi15} for the computation of Harish-Chandra induction.

\begin{thm}   \label{thm:exc}
 Let $G$ be simple of exceptional Lie type in characteristic $r$ and $p$ a
 prime dividing $|G|_{r'}$. Then $G$ has property $(I_p)$ if and only if it
 occurs in Table~$\ref{tab:exc}$.
\end{thm}

\begin{table}[htb]
\caption{Exceptional groups of Lie type satisfying $(I_p)$}   \label{tab:exc}
$\begin{array}{cccc}
 G& H& p& \dim\Phi_{1_G}\\
\hline
 \tw2B_2(q^2),\ q^2\ge8&           P& p|(q^4+1)& q^4+1\\
   ^2G_2(q^2),\ q^2\ge27&           P& 2<p|(q^6+1)& q^6+1\\
   ^2G_2(q^2),\ q^2\ge27&      O^2(B)& 2& 2(q^6+1)\\
  \tw2F_4(2)'& \PSL_3(3).2& 5& 1600\\
\end{array}$
\end{table}

\begin{proof}
We discuss the various families in turn. If Sylow $p$-subgroups of $G$ are
cyclic, we can argue using Lemma~\ref{lem:cyc} in conjunction with \cite{C81}.
This leads to the first two entries in Table~\ref{tab:exc}.

For the Suzuki groups, all Sylow subgroups for non-defining primes are cyclic.
For the Ree groups $^2G_2(q^2)$, $q^2=3^{2f+1}\ge27$, the only non-cyclic Sylow
subgroups for non-defining primes are for $p=2$. Here, Fong \cite{Fong} has
shown that $\Phi_{1_G}$ is induced from $O^2(B)$, for $B$ a Borel subgroup.
\par
For the other series we make use of the result of Liebeck and Saxl mentioned
earlier. For $G=G_2(q)$, $q>2$, the relevant primes are the divisors of $q^2-1$.
The cases $q=3,4$ can be dealt with via \GAP, so assume $q\ge5$.
If $2<p|(q+1)$, then a Borel subgroup $B$ of $G$ has order $q^6(q-1)^2$ prime
to~$p$. The only maximal subgroups of larger order, by \cite[Tab.~1]{LS87}, are
the two types of maximal parabolic subgroups (which both contain a Borel
subgroup of $G$) as well as $\SL_3(q).2$ and
$\SU_3(q).2$. The largest $p'$-subgroups of the latter two have order less
than $|B|$ by Proposition~\ref{prop:large SLn}, while (by Harish-Chandra
theory) both parabolic subgroups have a defect zero constituent in their
permutation character, so $H$ is contained in neither, and no example can arise.

Now assume $2<p|(q-1)$. Here a Sylow $r$-subgroup $U$ of $G$ gives a lower bound
$q^6$ for~$|H|$. Again, the maximal subgroups $\SL_3(q).2$ and $\SU_3(q).2$ do
not
have large enough $p'$-subgroups, and the same holds for $^2G_2(q)$ (if $q$ is
an odd power of~3) and $G_2(q^{1/2})$ (if $q$ is a square). So by \cite{LS87},
$H$ must lie in a maximal parabolic subgroup, of structure $[q^5].\GL_2(q)$.
By Proposition~\ref{prop:not parabolic} this forces $p=3$, and then $q=4$ by
Corollary~\ref{cor:compfac} and Proposition~\ref{prop:L2}, which was excluded
before.
\par
Finally, for $p=2$ the decomposition matrices in \cite[\S2.2.1 and 2.3.1]{HS92}
show that we do not get an example.   \par
For $G=\tw3D_4(q)$, the relevant primes are the divisors of $q^6-1$. For
$2<p$ dividing $q^3+1$, a Borel subgroup $B$ of $p'$-order $q^{12}(q-1)(q^3-1)$
shows that $H$ must lie in a maximal parabolic subgroup, by \cite{LS87}. But
their permutation characters contain constituents from non-principal blocks.
For $2<p$ dividing $q^3-1$ a potential subgroup $H$ must again be contained in
a maximal parabolic subgroup. In this case \cite[Prop.~5.3]{Ge91} shows that
the permutation characters properly contain the 1-PIM.
For $p=2$ we do not get an example by the decomposition matrix in \cite{Hi07}.
\par
For $G=\tw2F_4(q^2)'$, $q^2=2^{2f+1}\ge2$, the case $q^2=2$ can be treated
directly, leading to the stated example with $p=5$. For $q^2\ge8$ we need to
discuss primes~$p$ dividing $q^8-1$. Again, by \cite{Ma91}, $H$ must be
contained in a maximal parabolic subgroup $P$. The permutation characters on
both maximal parabolic subgroups contain constituents of $p$-defect zero, for
any divisor~$p$ of $(q^2+1)(q^4+1)$. So assume $p|(q^2-1)$.  Here
\cite[Thm~3.1]{Hi11} shows
that the trivial character is the only unipotent constituent of the 1-PIM, and
hence the 1-PIM is a proper summand of the permutation character on $P$.
\par
For $G=F_4(q)$ the relevant primes are the divisors of $(q^6-1)(q^2+1)$. By
Propositions~\ref{prop:large SO_2n+1} and~\ref{prop:large SO_2n+} any
$p'$-subgroup of the maximal subgroups of types $B_4(q)$, $D_4(q).\fS_3$ or
$\tw3D_4(q).3$ of $G$ has order less than $q^{24}$, the order of a Sylow
$r$-subgroup of $G$,
or is parabolic, hence contained in a parabolic subgroup of $G$ by the
Borel--Tits theorem. The same holds for the subfield subgroups. So we just need
to discuss parabolic subgroups. Now by Harish-Chandra theory, for any $p$
dividing $(q^6-1)(q^2+1)$ but not $q-1$, the permutation characters
on all maximal parabolic subgroups of $G$ have constituents in non-principal
$p$-blocks. So only primes dividing $q-1$ remain. By
Proposition~\ref{prop:not parabolic}, we may in fact assume that $p=2$ or
$p=3$. Then by Theorems~\ref{thm:psl}, \ref{thm:PSp} and~\ref{thm:PSO 2n+1}
the only Levi subgroups of a maximal parabolic subgroup having property $(I_p)$
are the ones of type $A_2+A_1$ for $p=3$ when $q=4$. Now the Harish-Chandra
induction of the 1-PIM from one of the two Levi subgroups of type $A_2+A_1$
contains the unipotent character $\phi_{9,2}$ four times, while the
Harish-Chandra induction of the 1-PIM from a Levi subgroup of type $C_3$ only
contains it twice. Thus, the former cannot contain a $p'$-subgroup with respect
to which $G$ satisfies $(I_3)$. The other potential maximal parabolic
subgroup is now also ruled out by application of the graph automorphism.
\par
For $G=E_6(q)$, arguing as before we need to discuss primes dividing $q^6-1$.
The candidates for maximal parabolic subgroups $P$ satisfying $(I_p)$ can be
read off from Theorems~\ref{thm:psl} and~\ref{thm:PSO 2n}. For $P$ of type
$A_5$ or $A_4+A_1$ the only possibility is $e_p(q)=6$, but in the first
case, the Harish-Chandra induction of the 1-PIM to $G$ contains the unipotent
character $\phi_{15,5}$, and in the second, the permutation character contains
$\phi_{64,4}$. Since both lie outside the principal $p$-block, these cases are
out. For $P$ of type $2A_2+A_1$, the constituent $\phi_{10,9}$ in the
permutation character rules out the possibility $e_p(q)=6$, and the constituent
$\phi_{81,6}$ shows we can't have $e_p(q)=3$ or $p=3$. This exhausts the
candidate maximal parabolic subgroups. All of the non-parabolic maximal
subgroups
listed in \cite[Tab.~1]{LS87} have order divisible by any prime $p$ for which
Sylow $p$-subgroups of $G$ are non-cyclic, but none of them has property
$(I_p)$ by Theorems~\ref{thm:psl}, \ref{thm:PSO 2n} and by induction, except
for the subsystem subgroup of type $A_5+A_1$ for $e_p(q)=6$. But by
Proposition~\ref{prop:large SLn} the largest $p'$-subgroup of the latter has
order smaller than $q^{36}$, the order of a maximal unipotent subgroup of $G$.
\par
For $G=\tw2E_6(q)$, the maximal parabolic subgroups possibly having
property $(I_p)$ are those of type $\tw2A_5(q)$ for $p=2$; those of type
$\tw2A_5(q)$ and $A_2(q^2)A_1(q)$ for $p=3,q=2$; and those of type
$A_2(q)A_1(q^2)$ for $p=3,q=4$, by Theorems~\ref{thm:psl}, \ref{thm:PSU}
and~\ref{thm:PSO 2n}.
The permutation characters of the parabolic subgroups in the last two cases
contain the unipotent character $\phi_{8,3}'$, outside the principal 3-block by
\cite{En00}. Now assume $p=2$ and $P$ is parabolic of type $\tw2A_5(q)$. Here,
the unipotent part of the induced 1-PIM of $P$ properly contains the unipotent
part of the induced 1-PIM from a parabolic subgroup of type $A_2(q^2)A_1(q)$
and thus $P$ cannot contain an admissible $p'$-subgroup. The non-parabolic
subgroups in \cite[Tab.~1]{LS87} all have order divisible by $p$, but none has
$(I_p)$ by Theorems~\ref{thm:PSU}, and~\ref{thm:PSO 2n}, except possibly the
subsystem subgroup of type $\tw2A_5(q)A_1(q)$ for $(p,q)=(3,2)$ or for $p=2$.
Here again the largest $p'$-subgroup of the latter has order smaller
than~$q^{36}$.
\par
For $G=E_7(q)$, the relevant primes are the divisors $p$ of $(q^6-1)(q^2+1)$.
The maximal parabolic subgroups possibly having $(I_p)$ are those of types
$A_5+A_1$, $A_3+A_2+A_1$ and $A_4+A_2$ for $e_p(q)=6$, and
$A_3+A_2+A_1$ for $e_p(q)=4$, by Theorems~\ref{thm:psl} and~\ref{thm:PSO 2n}.
For all of these, the permutation character has a constituent not lying in the
principal $p$-block. All large non-parabolic maximal subgroups of $G$ from
\cite[Tab.~1]{LS87} have order divisible by $p$, but none of them has
property~$(I_p)$ by our earlier results.
\par
For $G=E_8(q)$, Sylow $p$-subgroups are non-cyclic for primes $p$ such that
the Euler $\varphi$-function of $e_p(q)$ is at most~4. For all maximal
parabolic subgroups the permutation character contains constituents outside the
principal $p$-block when $e_p(q)>3$. Furthermore, by our earlier results, none
of these has a Levi subgroup with $(I_p)$ for $e_p(q)\le3$, so parabolic
subgroups cannot contain
an admissible $p'$-subgroup. Of the non-parabolic subgroups in
\cite[Tab.~1]{LS87}, only $H=E_8(\sqrt{q})$ might lead to an example, for
$e_p(q)=8$. But the order of $H$ is smaller than that of the $p'$-parabolic
subgroup of type $A_6+A_1$, so it cannot lead to an example
by Lemma~\ref{lem:order}.
\end{proof}

\section{About groups of Lie type in defining characteristic} \label{sec:defchar}

At present we are not able to settle our question for simple groups of
Lie type when $p$ is the defining characteristic.
The only examples we are aware of are the ones for $\PSL_2(q)$ in
Proposition~\ref{prop:L2}(4)--(6), $\fA_5\cong\PSL_2(4)\cong\PSL_2(5)$ for
$p=2$ and $p=5$, $\fA_6\cong\PSp_4(2)$ for $p=2$, $\PSL_3(2)$ for $p=2$,
and the group $\PSU_4(2)\cong\PSp_4(3)$ for $p=2$ and $p=3$.

We would not be surprised if these turn out to be the only ones.
A heuristic argument for this runs as follows. First, the Steinberg
character has trivial constituents upon restriction to small enough
subgroups. Observe that the bound for type $A_n$ is the same as the one
implied by Lemma~\ref{lem:l(G)}.

\begin{prop}   \label{prop:St}
 Let $G=G(q)$ be simple of Lie type in characteristic~$p$. Then $G$ does not
 have property $(I_p)$ with respect to any $p'$-subgroup of order at most $q^m$,
 with $m=m(G)$ as in Table~\ref{tab:m}.
\end{prop}

\begin{table}[htb]
\caption{Bounds for $p'$-subgroups}   \label{tab:m}
$\begin{array}{c|ccccccccc}
 G& A_n\ (n\ge1)& B_n,C_n\ (n\ge2)& D_n\ (n\ge4)& G_2& \tw3D_4& F_4& E_6,\tw2E_6& E_7& E_8\\
 \hline
 m(G)& n& 2(n-1)& 2(n-1)& 3& 8& 8& 16& 32& 56\\
\end{array}$
\end{table}

\begin{proof}
Let $H\le G$. We evaluate the scalar product of the Steinberg character $\St$
of $G$ restricted to $H$ with the trivial
character of $H$. The value of $\St$ on a semisimple element $s\in G$ is up to
sign equal to $|C_G(s)|_p$ (see e.g.~\cite[Prop.~3.4.10]{GM20}), which in turn
equals $q^N$ for $N$ the number of positive roots in the root system of the
underlying algebraic group $\mathbf G$, and zero on all other elements. Since
centralisers of semisimple elements are maximal rank subgroups which can be
seen on the extended Dynkin diagram of $\mathbf G$ (see e.g. \cite[\S13]{MT}),
it is easy to determine an upper bound $m$ such
that $|\St(s)|q^m\le \St(1)$ for any $1\ne s\in G$. For example, if $G$ is
of type $B_n$ with $n\ge2$, then $\St(1)=q^{n^2}$ while the maximal rank
subgroups with largest Sylow $p$-subgroup are of type $B_{n-1}B_1$, so we have
$m=n^2-(n-1)^2-1=2(n-1)$. For $G$ of type $E_8$, $\St(1)=q^{120}$, and
maximal Sylow $p$-subgroups are attained for subgroups of type $E_7A_1$,
giving $m=120-63-1=56$. In type $\tw3D_4$ the relevant subgroups are involution
centralisers of type $A_1(q^3)A_1(q)$.   \par
Then
$$|H|\,\langle\St|_H,1_H\rangle\ge \St(1)-(|H|-1)\St(1)/q^m>0.$$
Thus, $\St|_H$ has trivial constituents when $|H|\le q^m$. Since the Steinberg
character is of $p$-defect zero, it gives rise to a simple $kG$-module, and so
for $H$ a $p'$-subgroup, $G$ cannot have property $(I_p)$ with respect to $H$
by Lemma~\ref{lem:fp}.
\end{proof}

Secondly, generically any $p'$-subgroup of $G$ should have order bounded
above by $q^{m(G)}$ and Lemma~\ref{lem:order} would allow to conclude.

We have implemented this approach for the exceptional groups of small rank.

\begin{lem}
 The Suzuki groups $G=\tw2B_2(q^2)$ with $q^2=2^{2f+1}\ge8$ do not satisfy
 $(I_2)$.
\end{lem}

\begin{proof}
According to the description by Suzuki, the largest odd order subgroups of~$G$
are cyclic of order $q^2+\sqrt{2}q+1$ (\cite[Tab.~8.16]{BHR}). The restriction
of the Steinberg character of $G$ to such a subgroup $H$ contains the trivial
character with multiplicity $q^2-\sqrt{2}q>0$, so $G$ cannot have
property $(I_2)$ with respect to $H$ by Lemma~\ref{lem:fp}, and hence neither
with respect to any other $p'$-subgroup by Lemma~\ref{lem:order}.
\end{proof}

\begin{lem}
 The Ree groups $G={}^2G_2(q^2)$ with $q^2=3^{2f+1}\ge27$ do not satisfy $(I_3)$.
\end{lem}

\begin{proof}
From \cite[Tab.~8.43]{BHR} it follows that the largest order $3'$-subgroups
$H$ of $G$ are direct products of a dihedral group of order $q^2+1$ with
a Klein four group. An easy calculation shows that the trivial character
occurs with positive multiplicity in the restriction to $H$ of the Steinberg
character of $G$ and we conclude as in the previous case.
\end{proof}

\begin{lem}
 The groups $G=G_2(q)$ with $q=p^f\ge3$ do not satisfy $(I_p)$.
\end{lem}

\begin{proof}
The cases with $q\le5$ are out using \GAP.
According to the lists in \cite[Tab.~8.30, 8.41 and~8.42]{BHR} the maximal
order $p'$-subgroups $H$ of $G$ are among the torus normaliser of order
$12(q+1)$ and then $2^3.\PSL_3(2)$, $\PSL_2(13)$, $\PSL_2(8)$ and $G_2(2)$.
The bound obtained in Proposition~\ref{prop:St} is not quite large enough to
exclude these for all $q>5$, so we refine the argument.   \par
Namely, the
only elements $s\in G$ with $|C_G(s)|_p=q^3$ are elements of order~3, so only
such elements of $H$ can contribute $-q^3$ to the scalar product. The only
elements with $|C_G(s)|_p=q^2$ are involutions, but here $C_G(s)$ has type
$A_1^2$ of $\FF_q$-rank~2, so $\St(s)>0$. All other elements $1\ne s\in G$
have $|\St(s)|\le q$. With these improved estimates it is straightforward to
check that the four individual groups listed above are not relevant for $q>5$.
Finally, at most $2(q+1)^2$ of the elements in the torus normaliser $H$ have
order~3 (unless $p=2$ in which case we only consider a maximal odd order
subgroup). Then using the above argument one sees that again $H$ is not
relevant for $q>5$. 
\end{proof}

\begin{lem}
 The groups $G=\tw3D_4(q)$ with $q=p^f$ do not satisfy $(I_p)$.
\end{lem}

\begin{proof}
Using \GAP\ we may assume $q\ge3$. By \cite[Tab.~8.51]{BHR} $p'$-subgroups of
maximal order are among the torus normaliser $(q^2+q+1)^2.\SL_2(3)$
(respectively an odd order subgroup $(q^2+q+1)^2.3$ if $p=2$), and the
$p'$-subgroups of $G_2(q)$. The claim follows immediately from
Proposition~\ref{prop:St}.
\end{proof}

\begin{lem}
 The groups $G=\tw2F_4(q^2)'$ with $q^2=2^{2f+1}\ge2$ do not satisfy $(I_2)$.
\end{lem}

\begin{proof}
Any odd order subgroup $H$ of $G$ is (of course) solvable and hence local, and
lies inside torus normalisers. From \cite{Ma91} it follows that
$|H|\le 3(q^2+\sqrt{2}q+1)^2$. Using \GAP\ to deal with the case $q^2=2$ we
can then conclude with Proposition~\ref{prop:St}.
\end{proof}

\begin{lem}
 The groups $G=F_4(q)$ with $q=p^f$ do not satisfy $(I_p)$.
\end{lem}

\begin{proof}
Again the decomposition matrix of $F_4(2)$ is available in \GAP\ and shows the
dimension of the 1-PIM does not divide the group order. So we assume $q\ge3$.
Candidates for maximal $p'$-subgroups apart from torus normalisers are
determined in \cite{Cr21}. All of these have orders divisible by 6, and for
$q\ge5$ they are ruled out by Proposition~\ref{prop:St}. The same result
shows that the maximal order torus normaliser of type $(q+1)^4.W(F_4)$
is too small, using that only involutions $s\in G$ attain $|C_G(s)|_p=q^{16}$
while all other semisimple elements $1\ne s\in G$ have
$|\St(s)|=|C_G(s)|_p\le q^{10}$.
\end{proof}

While we expect that with some more work the remaining exceptional groups can
also be dealt with along those same lines, the bound in
Proposition~\ref{prop:St} is definitely too weak to handle the linear and
unitary groups.


\end{document}